\documentclass[reqno]{amsart}
\usepackage{amsmath}
\usepackage{amssymb}
\usepackage{amsthm}
\usepackage{tikz,tikz-cd}
\usepackage{subfig}
\usepackage{mathrsfs}

\usepackage{hyperref}

\usetikzlibrary{positioning}
\usetikzlibrary{arrows}

\title{A local characterization of $B_2$ regular crystals}

\author{Shunsuke Tsuchioka}
\address{Graduate School of Mathematical Sciences, University of Tokyo,
Komaba, Meguro, Tokyo, 153-8914, Japan}
\email{tshun@kurims.kyoto-u.ac.jp}

\date{Oct 30, 2019}

\newtheorem{Thm}{Theorem}[section]
\newtheorem{Def}[Thm]{Definition}

\newtheorem{Prop}[Thm]{Proposition}
\newtheorem{Lem}[Thm]{Lemma}

\newtheorem{Rem}[Thm]{Remark}
\newtheorem{Cor}[Thm]{Corollary}
\newtheorem{Ex}[Thm]{Example}

\newcommand{\TRANS}[1]{{}^{{t}}{#1}}
\newcommand{\CHECK}{\vee}

\newcommand{\ISOM}{\stackrel{\sim}{\longrightarrow}}
\newcommand{\ZERO}{\boldsymbol{0}}
\newcommand{\Z}{\mathbb{Z}}

\DeclareMathOperator{\HOM}{Hom}
\DeclareMathOperator{\RANK}{rank}

\newcommand{\EFU}{h}
\newcommand{\KE}[1]{\tilde{e}_{#1}}
\newcommand{\KF}[1]{\tilde{f}_{#1}}
\newcommand{\KG}[1]{\tilde{g}_{#1}}
\newcommand{\LusR}{R}
\newcommand{\LusRR}{R^{-1}}
\DeclareMathOperator{\WT}{wt_0}
\DeclareMathOperator{\WTT}{wt}
\DeclareMathOperator{\WTTT}{\mathsf{wt}}

\DeclareMathOperator{\DIST}{depth}

\begin{document}
\maketitle

\begin{abstract}
Stembridge characterized regular crystals associated with a simply-laced generalized Cartan matrix (GCM) in terms of local graph-theoretic quantities.
We give a similar axiomatization for $B_2$ regular crystals and thus for regular crystals associated with a 
finite GCM except $G_2$ and 
an affine GCM except $A^{(1)}_{1},G^{(1)}_{2},A^{(2)}_{2},D^{(3)}_4$.
\end{abstract}

\section{Introduction}\label{intro}

\subsection{Kashiwara crystals}\label{defkc}
Let $A=(a_{ij})_{i,j\in I}$ be a symmetrizable generalized Cartan matrix (GCM) in the sense of ~\cite[\S1.1,\S2.1]{Kac}.
In this paper, $I$ is always assumed to be a finite set.
We fix a Cartan datum for $A$, which is a 4-tuple $(P,P^{\CHECK},\Pi,\Pi^{\CHECK})$ such that
\begin{enumerate}
\item $P$ is a free $\Z$-module of rank $2|I|-\RANK A$ and $P^{\CHECK}=\HOM_{\Z}(P,\Z)$,
\item $\Pi=\{\alpha_i\mid i\in I\}$ (resp. $\Pi^{\CHECK}=\{h_i\mid i\in I\}$) are $\Z$-linearly independent elements in $P$ (resp. $P^{\CHECK}$) such that
$a_{ij}=\langle h_i,\alpha_j\rangle$ for all $i,j\in I$, where $\langle,\rangle:P^\vee\times P\to\Z$ is the canonical pairing.
\end{enumerate}


A Kashiwara crystal is a 6-tuple $(B,\WTTT,({\KE{i}})_{i\in I},({\KF{i}})_{i\in I},(\varepsilon_i)_{i\in I},(\varphi_i)_{i\in I})$, 
where $B$ is a set and $\WTTT:B\to P,\varepsilon_i,\varphi_i:B\to\mathbb{Z}\sqcup\{-\infty\},\KE{i},\KF{i}:B\to B\sqcup\{\ZERO\}$ are functions
that satisfy the axioms (K1)--(K5) given below (see also ~\cite[\S7.2]{Kas3}). In (K1)--(K3), we understand $a+(-\infty)=(-\infty)+a=-\infty$ for $a\in\Z$.
\begin{enumerate}
\item[(K1)] $\forall i\in I,\forall b\in B,\varphi_i(b)=\varepsilon_i(b)+\langle h_i,\WTTT(b)\rangle$.
\item[(K2)] \mbox{$\forall i\in I,\forall b\in B,\KE{i}b\ne\ZERO\Rightarrow \WTTT(\KE{i}b)=\WTTT(b)+\alpha_i,\varepsilon_i(\KE{i}b)=\varepsilon_i(b)-1,\varphi_i(\KE{i}b)=\varphi_i(b)+1$}.
\item[(K3)] \mbox{$\forall i\in I,\forall b\in B,\KF{i}b\ne\ZERO\Rightarrow \WTTT(\KF{i}b)=\WTTT(b)-\alpha_i,\varepsilon_i(\KF{i}b)=\varepsilon_i(b)+1,\varphi_i(\KF{i}b)=\varphi_i(b)-1$}.
\item[(K4)] $\forall i\in I,\forall b\in B,\forall b'\in B,\KE{i}b=b'\Leftrightarrow b=\KF{i}b'$.
\item[(K5)] $\forall i\in I,\forall b\in B,\varphi_i(b)=-\infty\Rightarrow \KE{i}b=\KF{i}b=\ZERO$.
\end{enumerate}

\begin{Rem}\label{shorthand}
In this paper, we use shorthand such as $\forall,\exists,\Leftrightarrow$ when we state axioms
because the usage has some advantage in that
the presentation becomes unambiguous, mathematically formal and more translatable to a computer implementation.
\end{Rem}

\subsection{Highest weight crystals and regular crystals}
For a dominant integral weight $\lambda\in P^+$ (i.e., $\forall i\in I,\langle h_i,\lambda\rangle\geq 0$), 
Kashiwara established the existence and uniqueness of the crystal basis $B(\lambda)$ (called 
the highest weight crystal) of the 
integrable highest weight representation $V(\lambda)$ of the quantum group $U_q(A)$~\cite[Theorem 2]{Kas2}.
Because Kashiwara's theorems are proved using a power of abstraction of representation theory,
some prefer 
elementary approaches.

A Kashiwara crystal 
is called \emph{regular}(see ~\cite[\S2.2]{Kas4}) if 
it is isomorphic as $A|_{i,j}$-crystals to the crystal basis of an integrable $U_q(A|_{i,j})$-module for any
$i\ne j\in I$ such that 
\begin{align*}
A|_{i,j}:=\begin{pmatrix} a_{ii} & a_{ij} \\ a_{ji} & a_{jj}\end{pmatrix}=A_1\oplus A_1,A_2,B_2,{}^tB_2,G_2,{}^tG_2.
\end{align*}
Assuming a certain boundedness condition~\cite[(2.4.1)]{KKMMNN2}, 
regular crystal is a disjoint union of the highest weight crystals~\cite[Proposition 2.4.4]{KKMMNN2}.



\subsection{Crystal graphs}
A Kashiwara crystal $(B,\WTTT,({\KE{i}})_{i\in I},({\KF{i}})_{i\in I},(\varepsilon_i)_{i\in I},(\varphi_i)_{i\in I})$ gives an $I$-colored directed graph (the crystal graph of $B$) by the rule:
\begin{quotation}
there is an $i$-colored arrow from $x$ to $y$ if and only if $\KF{i}x=y$.
\end{quotation}

\begin{Def}\label{goodgraph}
An $I$-colored directed graph $X$ is \emph{good} if for any $x\in X$ and $i\in I$
\begin{enumerate}
\item[(G1)] there is at most one $i$-colored arrow from $x$, 
\item[(G2)] there is at most one $i$-colored arrow to $x$, 
\item[(G3)] the length of the $i$-string through $x$ is finite.
\end{enumerate}
\end{Def}

When there is an $i$-colored arrow from $x$ to $y$ in a good $I$-colored directed graph $X$, we define as $\KF{i}x=y$ and $\KE{i}y=x$.
$\KF{i}x=\ZERO$ (resp. $\KE{i}x=\ZERO$) means that
there is no $i$-colored arrow from $x$ (resp. to $x$).
Thanks to the axioms, 
\begin{align*}
\varphi_i(x) = \max\{m\geq 0\mid \KF{i}^mx\ne\ZERO\},\quad
\varepsilon_i(x) = \max\{m\geq 0\mid \KE{i}^mx\ne\ZERO\}
\end{align*}
are well-defined (i.e., take finite values).
The crystal graph of a highest weight crystal $B(\lambda)$ is good
and the quantities $\varepsilon_i,\varphi_i$ are the same as above~\cite[(2.4.1)]{Kas2}.

\begin{Def}\label{maxelem}
Let $X$ be a good $I$-colored directed graph. We say that $x_0\in X$ is a \emph{maximum element} if
\begin{enumerate}
\item[(M1)] $\forall i\in I,\KE{i}x_0=\ZERO$ (i.e., $\varepsilon_i(x_0)=0$),
\item[(M2)] $\forall x\in X,\exists s\geq 0,\exists (i_1,\cdots,i_s)\in I^s,\KF{i_1}\cdots\KF{i_s}x_0=x$ (i.e., $\KE{i_s}\cdots\KE{i_1}x=x_0$).
\end{enumerate}
\end{Def}

\begin{Def}\label{quanef}
Let $X$ be a good $I$-colored directed graph. We define 
\begin{align*}
\Delta^g_{\beta}(i,j,x) &= \beta_j(\KG{i}x)-\beta_j(x)
\end{align*}
for $g\in\{e,f\}$, $\beta\in\{\varepsilon,\varphi\}$ and $x\in X$, $i,j\in I$
with $\KG{i}x\ne\ZERO$.
\end{Def}

\subsection{Stembridge crystals}\label{stc}
For a symmetrizable GCM $A$,
Stembridge proved the $A$-regularity of highest weight crystals via Littelmann's path model~\cite{Li1,Li2}. 

\begin{Thm}[{\cite[Definition 1.1, Proposition 2.4]{Ste}}]\label{StembridgeCrystal}
Let $A=(a_{ij})_{i,j\in I}$ be a symmetrizable GCM. 
For a dominant integral weight $\lambda\in P^+$,
the highest weight crystal $B(\lambda)$ is an $A$-regular graph 
having a maximum element $b_{\lambda}\in B(\lambda)$ with $\varphi_i(b_{\lambda})=\langle h_i,\lambda\rangle$ for all $i\in I$.
An $A$-regularity of a directed graph $X$ is defined by the axioms (S1)--(S5) given below (see also Remark \ref{shorthand}).
\end{Thm}

\begin{enumerate}
\item[(S1)] $X$ is a good $I$-colored directed graph in the sense of Definition \ref{goodgraph}.
\item[(S2)] $\forall x\in X,\forall i\in I,\KE{i}x\ne\ZERO \Rightarrow \forall j\in I\setminus\{i\},\Delta^e_{\varphi}(i,j,x)-\Delta^e_{\varepsilon}(i,j,x)=a_{ji}$.
\item[(S3)] $\forall x\in X,\forall i\in I,\KE{i}x\ne\ZERO \Rightarrow \forall j\in I\setminus\{i\},\Delta^e_{\varphi}(i,j,x)\leq 0\leq\Delta^e_{\varepsilon}(i,j,x)$.
\item[(S4)] $\forall i\ne\forall j\in I,\forall x\in X$,$\KE{i}x\ne\ZERO\ne\KE{j}x\Rightarrow$ (A$^-_{i,j}$),(A$^-_{j,i}$),(B$^-$).
\item[(S5)] $\forall i\ne\forall j\in I,\forall x\in X$,$\KF{i}x\ne\ZERO\ne\KF{j}x\Rightarrow$ (A$^+_{i,j}$),(A$^+_{j,i}$),(B$^+$).
\end{enumerate}

\noindent{(A$^-_{k,\ell}$)} 
\mbox{$\Delta^e_{\varepsilon}(k,\ell,x)=0\Rightarrow\exists z=\KE{\ell}\KE{k}x=\KE{k}\KE{\ell}x,\Delta^f_{\varphi}(\ell,k,z)=0$}.

\noindent{(B$^-$)} 
\mbox{$(\Delta^e_{\varepsilon}(i,j,x),\Delta^e_{\varepsilon}(j,i,x))=(1,1)\Rightarrow\exists z=\KE{i}\KE{j}^2\KE{i}x=\KE{j}\KE{i}^2\KE{j}x,(\Delta^f_{\varphi}(i,j,z),\Delta^f_{\varphi}(j,i,z))=(1,1)$}.

\noindent{(A$^+_{k,\ell}$)} 
\mbox{$\Delta^f_{\varphi}(k,\ell,x)=0\Rightarrow\exists z=\KF{\ell}\KF{k}x=\KF{k}\KF{\ell}x,\Delta^e_{\varepsilon}(\ell,k,z)=0$}.

\noindent{(B$^+$)} 
\mbox{$(\Delta^f_{\varphi}(i,j,x),\Delta^f_{\varphi}(j,i,x))=(1,1)\Rightarrow\exists z=\KF{i}\KF{j}^2\KF{i}x=\KF{j}\KF{i}^2\KF{j}x,(\Delta^e_{\varepsilon}(i,j,z),\Delta^e_{\varepsilon}(j,i,z))=(1,1)$}.

\vspace{2mm}

\noindent{\bf Convention.} 
In the definition, for example, $\exists z=\KE{i}\KE{j}^2\KE{i}x=\KE{j}\KE{i}^2\KE{j}x$ in (B$^-$) means
\begin{enumerate}
\item $\KE{i}x,\KE{j}\KE{i}x,\KE{j}\KE{j}\KE{i}x,\KE{i}\KE{j}\KE{j}\KE{i}x,\KE{j}x,\KE{i}\KE{j}x,\KE{i}\KE{i}\KE{j}x,\KE{j}\KE{i}\KE{i}\KE{j}x\ne\ZERO$,
\item $\KE{i}\KE{j}\KE{j}\KE{i}x=\KE{j}\KE{i}\KE{i}\KE{j}x$ and we name it $z$.
\end{enumerate}
\begin{Rem}\label{abb}
As in ~\cite[p.4810]{Ste}, (1) has a redundancy in that
some are forced automatically 
(i.e., $\KE{i}x,\KE{j}\KE{i}x,\KE{j}\KE{j}\KE{i}x,\KE{j}x,\KE{i}\KE{j}x,\KE{i}\KE{i}\KE{j}x\ne\ZERO$ follows from $(\Delta^e_{\varepsilon}(i,j,x),\Delta^e_{\varepsilon}(j,i,x))=(1,1)$ and $\varepsilon_i(x),\varepsilon_j(x)\geq 1$).
We have similar remarks on (A$^{\pm}_{k,\ell}$).
However we will not consider minimization of axioms and
use freely abbreviations \mbox{involving $\exists$.} 
\end{Rem}

\begin{Rem}\label{henkanef}
Note $\Delta^e_{\varepsilon}(i,i,x)=-1,\Delta^e_{\varphi}(i,i,x)=1,\Delta^f_{\varepsilon}(i,i,x)=1,\Delta^f_{\varphi}(i,i,x)=-1$ if defined.
Also note $\Delta^f_{\beta}(i,j,x)=-\Delta^e_{\beta}(i,j,\KF{i}x)$ for $\beta\in\{\varepsilon,\varphi\}$ if defined.
Thus, we can replace (S2) (resp. (S3)) with (S2$_{k}$) for $k\in\{a,b,c\}$ (resp. (S3')). 
\begin{enumerate}
\item[(S2$_a$)] $\forall x\in X,\forall i\in I,\KE{i}x\ne\ZERO \Rightarrow \forall j\in I,\Delta^e_{\varphi}(i,j,x)-\Delta^e_{\varepsilon}(i,j,x)=a_{ji}$.
\item[(S2$_b$)] $\forall x\in X,\forall i\in I,\KF{i}x\ne\ZERO \Rightarrow \forall j\in I\setminus\{i\},\Delta^f_{\varepsilon}(i,j,x)-\Delta^f_{\varphi}(i,j,x)=a_{ji}$.
\item[(S2$_c$)] $\forall x\in X,\forall i\in I,\KF{i}x\ne\ZERO \Rightarrow \forall j\in I,\Delta^f_{\varepsilon}(i,j,x)-\Delta^f_{\varphi}(i,j,x)=a_{ji}$.
\item[(S3')] $\forall x\in X,\forall i\in I,\KF{i}x\ne\ZERO \Rightarrow \forall j\in I\setminus\{i\},\Delta^f_{\varepsilon}(i,j,x)\leq 0\leq\Delta^f_{\varphi}(i,j,x)$.
\end{enumerate}
\end{Rem}


When $A$ is simply-laced, (S2), (S3) imply $(\Delta^e_{\varepsilon}(i,j,x),\Delta^e_{\varepsilon}(j,i,x))=(0,0),(0,1),(1,0),(1,1)$ (resp. $(\Delta^f_{\varphi}(i,j,x),\Delta^f_{\varphi}(j,i,x))=(0,0),(0,1),(1,0),(1,1)$) if $\KE{i}x\ne\ZERO\ne\KE{j}x$ (resp. $\KF{i}x\ne\ZERO\ne\KF{j}x$) with $i\ne j$.
Recall that $A=(a_{ij})_{i,j\in I}$ is simply-laced if
\begin{align*}
\forall i\ne\forall j\in I, 
A|_{i,j}:=\begin{pmatrix} a_{ii} & a_{ij} \\ a_{ji} & a_{jj} \end{pmatrix}=
\underbrace{\begin{pmatrix} 2 & 0 \\ 0 & 2 \end{pmatrix}}_{A_1\oplus A_1},
\underbrace{\begin{pmatrix} 2 & -1 \\ -1 & 2 \end{pmatrix}}_{A_2}.
\end{align*}


Stembridge's $A$-regularity characterizes simply-laced highest weight crystals.
\begin{Thm}[{\cite[Proposition 1.4, Theorem 3.3]{Ste}}]\label{StembridgeCrystal2}
Let $A=(a_{ij})_{i,j\in I}$ be a simply-laced GCM and
let $X$ be an $A$-regular graph with a maximum element $x_0\in X$.
Then, there exists a unique $I$-colored directed graph isomorphism between $X$ and $B(\lambda)$, where $\lambda\in P^+$ satisfies $\langle h_i,\lambda\rangle=\varphi_i(x_0)$ for all $i\in I$.
\end{Thm}

\begin{Ex}\label{fundamentalweight}
For a symmetrizable GCM $A=(a_{ij})_{i,j\in I}$ with $\det A\ne 0$, 
the fundamental weight $\Lambda_i$ is defined by $\langle h_j,\Lambda_i\rangle=\delta_{ij}$ for all $j\in I$.
The left (resp. right) figure below is the $A_2$-crystal $B(2\Lambda_1)$ (resp. $B(\Lambda_1+\Lambda_2)$) and
gives a visualization of (A$^{-}_{1,2}$) (resp. (B$^-$)).
Here, thick arrows are 1-arrows. 
\end{Ex}

\begin{center}
\scalebox{0.5}{
\begin{tikzpicture}[
  every node/.style = {
    draw, circle, 
    inner sep = 2pt
  },
  thick arrow/.style = {
    ->, -stealth, 
    ultra thick
  },
  thin arrow/.style = {
    ->, -stealth, thin
  },
  ]
\begin{scope}
  \node[] (L0) {};

  \node[right = of L0, label = {[above]\Huge{$z$}}] (L1){};

  \node[below right = of L1] (L2 0){};
  \node[above right = of L1] (L2 1){};

  \node[above right = of L2 0, label = {[below = 5pt]\Huge{$x$}}] (L3){};
  
  \node[right = of L3] (L4){};

  \draw[thick arrow] (L0) -- (L1);

  \draw[thin arrow] (L1) -- (L2 0);
  \draw[thick arrow] (L1) -- (L2 1);

  \draw[thick arrow] (L2 0) -- (L3);
  \draw[thin arrow] (L2 1) -- (L3);

  \draw[thin arrow] (L3) -- (L4);
\end{scope}
  \begin{scope}[xshift=9cm]
  \node[label = {[left]\Huge{$z$}}] (L0) {};

  \node[below right = of L0] (L1 0){};
  \node[above right = of L0] (L1 1){};

  \node[right = of L1 0] (L2 0){};
  \node[right = of L1 1] (L2 1){};
  
  \node[right = of L2 0] (L3 0){};
  \node[right = of L2 1] (L3 1){};

  \node[above right = of L3 0, label = {[right]\Huge{$x$}}] (L4){};
  
  \draw[thick arrow] (L0) -- (L1 0);
  \draw[thin arrow] (L0) -- (L1 1);

  \draw[thin arrow] (L1 0) -- (L2 0);
  \draw[thick arrow] (L1 1) -- (L2 1);
  
  \draw[thin arrow] (L2 0) -- (L3 0);
  \draw[thick arrow] (L2 1) -- (L3 1);

  \draw[thick arrow] (L3 0) -- (L4);
  \draw[thin arrow] (L3 1) -- (L4);
\end{scope}
\end{tikzpicture}}
\end{center}

\subsection{The main result}\label{maintheorems}
Though one can use diagram folding techniques (see ~\cite[\S5]{BS},~\cite{NS},~\cite{Kas1}),
It is still desirable to have local axioms of regular crystals. 
Thanks to ~\cite[Proposition 2.4.4]{KKMMNN2}, this question is reduced to the rank 2 cases.
The aim of this paper is to give an answer for $B_2=\begin{pmatrix} 2 & -2 \\ -1 & 2 \end{pmatrix}$ 
(see \S\ref{prevstu}, on previous studies).

\begin{Thm}\label{maintheorem}
Let $A=(a_{ij})_{i,j\in I}$ be a symmetrizable GCM with $\forall i\ne\forall j\in I, A|_{i,j}=A_1\oplus A_1,A_2,B_2,\TRANS{B_2}$
and let $X$ be an $A$-regular graph with a maximum element $x_0\in X$ that further satisfies 
\begin{align*}
\forall i\ne\forall j\in I,A|_{i,j}=B_2\Rightarrow \textup{(S6),(S7),(S8),(S9)}.
\end{align*}
Then, there exists a unique $I$-colored directed graph isomorphism between $X$ and $B(\lambda)$, where $\lambda\in P^+$ satisfies $\langle h_i,\lambda\rangle=\varphi_i(x_0)$ for all $i\in I$. 
\end{Thm}

\begin{enumerate}
\item[(S6)] \mbox{$\forall x\in X,\KE{i}x\ne\ZERO\ne\KE{j}x,\Delta(x)=(1,2)\Rightarrow$ (D$^-$)}.
\item[(S7)] \mbox{$\forall x\in X,\KF{i}x\ne\ZERO\ne\KF{j}x,\Delta'(x)=(1,2)\Rightarrow$ (D$^+$)}.
\item[(S8)] \mbox{$\forall x\in X,\KF{i}x\ne\ZERO\ne\KF{j}x,\Delta'(x)=(1,1),\varphi_i(x)\geq 2\Rightarrow$ (C$^+_1$)}.
\item[(S9)] \mbox{$\forall x\in X,\KF{i}x\ne\ZERO\ne\KF{j}x,\Delta'(x)=(0,2),\KF{j}\KF{i}^2x\ne\ZERO,\Delta^f_{\varphi}(j,i,\KF{i}^2x)=0\Rightarrow$ (C$^+_1$)}.
\end{enumerate}

\noindent{(D$^-$)} 
\mbox{$y:=\KE{i}^2\KE{j}x$,$\exists y'=\KE{i}^2\KE{j}^2\KE{i}x$,(P$^-_1$),(Q$^-_1$),(R$^-$),$(\Delta^f_{\varphi}(i,j,y),\Delta^f_{\varphi}(i,j,y'))\ne (1,0)$}.

\noindent{(D$^+$)} 
\mbox{$y:=\KF{i}^2\KF{j}x$,$\exists y'=\KF{i}^2\KF{j}^2\KF{i}x$,$(\Delta^e_{\varepsilon}(i,j,y),\Delta^e_{\varepsilon}(i,j,y'))=(0,1)\Rightarrow\exists z=\KF{j}\KF{i}^3\KF{j}^2\KF{i}x=\KF{i}\KF{j}^2\KF{i}^3\KF{j}x$}.

\noindent{(C$^+_1$)} 
\mbox{$\exists z=\KF{i}\KF{j}^2\KF{i}^2x=\KF{j}\KF{i}^3\KF{j}x$}.

\noindent{(P$^-_1$)} 
\mbox{$(\Delta^f_{\varphi}(i,j,y),\Delta^f_{\varphi}(i,j,y'))=(1,1)\Rightarrow\KF{j}y'=\KE{i}y,\Delta^f_{\varphi}(j,i,y')=1$}.


\noindent{(Q$_1^-$)} 
\mbox{$(\Delta^f_{\varphi}(i,j,y),\Delta^f_{\varphi}(i,j,y'))=(0,1)\Rightarrow\exists z=\KE{j}\KE{i}^3\KE{j}^2\KE{i}x=\KE{i}\KE{j}^2\KE{i}^3\KE{j}x,\Delta'(z)=(1,2)$}.

\noindent{(R$^-$)} 
\mbox{$(\Delta^f_{\varphi}(i,j,y),\Delta^f_{\varphi}(i,j,y'))=(0,0)\Rightarrow\KF{j}y'=\KE{i}y,\Delta^f_{\varphi}(j,i,y')=2,\Delta^f_{\varphi}(j,i,\KF{i}^2y')=0$}.




\hspace{0mm}

Here, $\Delta(x)=(\Delta^e_{\varepsilon}(i,j,x),\Delta^e_{\varepsilon}(j,i,x))$ and $\Delta'(w)=(\Delta^f_{\varphi}(i,j,w),\Delta^f_{\varphi}(j,i,w))$ for $w=x,z$.
Note that $y$ in (D$^-$) (resp. (D$^+$)) is just defined.
The existence is not a part of the axiom
because it follows from $\Delta^e_{\varepsilon}(j,i,x)=2$ (resp. $\Delta^f_{\varphi}(j,i,x)=2$).
Note also that we have $\KE{i}y\ne\ZERO$ in (P$^-_1$),(R$^-$) by $\Delta^e_{\varepsilon}(j,i,x)=2$ and $\varepsilon_i(x)\geq 1$. 

\begin{Ex}
We duplicate ~\cite[Figure 5]{Ste} as Figure \ref{crystalgraphs} (see also Example \ref{fundamentalweight}). 
In Figure \ref{crystalgraphs}, thick arrows are 1-arrows.
We can see an appearance of (Q$_1^-$),(P$^-_1$),(R$^-$) from left to right. 
We can also see (S7) (resp. (S8)) in the left (resp. middle) graph at $z$, and (S9) in the right graph at $y'$.
\end{Ex}

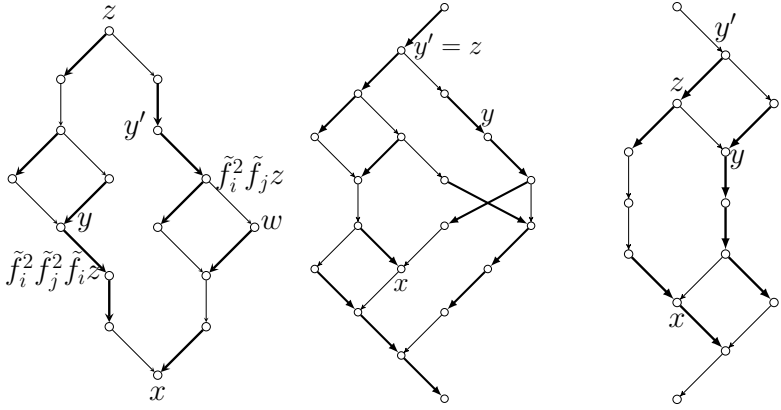
\begin{figure}
\begin{minipage}{.32\textwidth}
{\scalebox{0.56}{\begin{tikzpicture}[
  every node/.style = {
    draw, circle, 
    inner sep = 2pt
  },
  thick arrow/.style = {
    ->, -stealth, 
    ultra thick
  },
  thin arrow/.style = {
    ->, -stealth, thin
  },
  ]

  \node[, label = {[above]\Huge{$z$}}] (L0) {};

  \node[below left = of L0] (L1 0){};
  \node[below right = of L0] (L1 1){};

  \node[below = of L1 0] (L2 0){};
  \node[below = of L1 1, label = {[left]\Huge{$y'$}}] (L2 1){};

  \node[below left = of L2 0] (L3 0){};
  \node[below right = of L2 0] (L3 1){};
  \node[below right = of L2 1, label = {[right]\Huge{$\KF{i}^2\KF{j}z$}}] (L3 2){};

  \node[below left = of L3 1, label = {[right = 5pt]\Huge{$y$}}] (L4 0){};
  \node[below left = of L3 2] (L4 1){};
  \node[below right = of L3 2, label = {[right]\Huge{$w$}}] (L4 2){};

  \node[below right = of L4 0, label = {[left]\Huge{$\KF{i}^2\KF{j}^2\KF{i}z$}}] (L5 0){};
  \node[below right = of L4 1] (L5 1){};

  \node[below = of L5 0] (L6 0){};
  \node[below = of L5 1] (L6 1){};

  \node[below right = of L6 0, label = {[below = 5pt]\Huge{$x$}}] (L7){};

  \draw[thick arrow] (L0) -- (L1 0);
  \draw[thin arrow] (L0) -- (L1 1);

  \draw[thin arrow] (L1 0) -- (L2 0);
  \draw[thick arrow] (L1 1) -- (L2 1);

  \draw[thick arrow] (L2 0) -- (L3 0);
  \draw[thin arrow] (L2 0) -- (L3 1);
  \draw[thick arrow] (L2 1) -- (L3 2);

  \draw[thin arrow] (L3 0) -- (L4 0);
  \draw[thick arrow] (L3 1) -- (L4 0);
  \draw[thick arrow] (L3 2) -- (L4 1);
  \draw[thin arrow] (L3 2) -- (L4 2);

  \draw[thick arrow] (L4 0) -- (L5 0);
  \draw[thin arrow] (L4 1) -- (L5 1);
  \draw[thick arrow] (L4 2) -- (L5 1);

  \draw[thick arrow] (L5 0) -- (L6 0);
  \draw[thin arrow] (L5 1) -- (L6 1);

  \draw[thin arrow] (L6 0) -- (L7);
  \draw[thick arrow] (L6 1) -- (L7);
\end{tikzpicture}}}
\end{minipage}
\begin{minipage}{.32\textwidth}
{\scalebox{0.5}{\begin{tikzpicture}[
  every node/.style = {
    draw, circle, 
    inner sep = 2pt
  },
  thick arrow/.style = {
    ->, -latex, 
    ultra thick
  },
  thin arrow/.style = {
    ->, -stealth', thin
  },
  ]

  \node (L0) {};

  \node[below left = of L0, label = {[right = 5pt]\Huge{$y'=z$}}] (L1){};

  \node[below left = of L1] (L2 0){};
  \node[below right= of L1] (L2 1){};

  \node[below left = of L2 0] (L3 0){};
  \node[below right = of L2 0] (L3 1){};
  \node[below right = of L2 1, label = {[above]\Huge{$y$}}] (L3 2){};

  \node[below right = of L3 0] (L4 0){};
  \node[below right = of L3 1] (L4 1){};
  \node[below right = of L3 2] (L4 2){};

  \node[below = of L4 0] (L5 0){};
  \node[below = of L4 1] (L5 1){};
  \node[below = of L4 2] (L5 2){};

  \node[below left = of L5 0] (L6 0){};
  \node[below right= of L5 0, label = {[below = 5pt]\Huge{$x$}}] (L6 1){};
  \node[below left = of L5 2] (L6 2){};

  \node[below right= of L6 0] (L7 0){};
  \node[below left = of L6 2] (L7 1){};

  \node[below right= of L7 0] (L8){};

  \node[below right= of L8] (L9){};
  
  \draw[thick arrow] (L0) -- (L1);

  \draw[thick arrow] (L1) -- (L2 0);
  \draw[thin arrow] (L1) -- (L2 1);

  \draw[thick arrow] (L2 0) -- (L3 0);
  \draw[thin arrow] (L2 0) -- (L3 1);
  \draw[thick arrow] (L2 1) -- (L3 2);

  \draw[thin arrow] (L3 0) -- (L4 0);
  \draw[thick arrow] (L3 1) -- (L4 0);
  \draw[thin arrow] (L3 1) -- (L4 1);
  \draw[thick arrow] (L3 2) -- (L4 2);

  \draw[thin arrow] (L4 0) -- (L5 0);
  \draw[thin arrow] (L4 2) -- (L5 2);
  \draw[thick arrow] (L4 1) -- (L5 2);
  \draw[thick arrow] (L4 2) -- (L5 1);

  \draw[thin arrow] (L5 0) -- (L6 0);
  \draw[thick arrow] (L5 0) -- (L6 1);
  \draw[thin arrow] (L5 1) -- (L6 1);
  \draw[thick arrow] (L5 2) -- (L6 2);
  
  \draw[thick arrow] (L6 0) -- (L7 0);
  \draw[thin arrow] (L6 1) -- (L7 0);
  \draw[thick arrow] (L6 2) -- (L7 1);

  \draw[thick arrow] (L7 0) -- (L8);
  \draw[thin arrow] (L7 1) -- (L8);
  
  \draw[thick arrow] (L8) -- (L9);
\end{tikzpicture}}}
\end{minipage}
\begin{minipage}{.32\textwidth}
{\scalebox{0.56}{\begin{tikzpicture}[
  every node/.style = {
    draw, circle, 
    inner sep = 2pt
  },
  thick arrow/.style = {
    ->, -latex, 
    ultra thick
  },
  thin arrow/.style = {
    ->, -stealth', thin
  },
  ]

  \node (L0) {};

  \node[below right = of L0, label = {[above]\Huge{$y'$}}] (L1){};

  \node[below left = of L1, label = {[above]\Huge{$z$}}] (L2 0){};
  \node[below right= of L1] (L2 1){};

  \node[below left = of L2 0] (L3 0){};
  \node[below left = of L2 1, label = {[below right]\Huge{$y$}}] (L3 1){};

  \node[below = of L3 0] (L4 0){};
  \node[below = of L3 1] (L4 1){};

  \node[below = of L4 0] (L5 0){};
  \node[below = of L4 1] (L5 1){};

  \node[below right= of L5 0, label = {[below = 5pt]\Huge{$x$}}] (L6 0){};
  \node[below right= of L5 1] (L6 1){};

  \node[below right= of L6 0] (L7){};

  \node[below left= of L7] (L8){};
  
  \draw[thin arrow] (L0) -- (L1);

  \draw[thin arrow] (L1) -- (L2 1);
  \draw[thick arrow] (L1) -- (L2 0);

  \draw[thick arrow] (L2 0) -- (L3 0);
  \draw[thin arrow] (L2 0) -- (L3 1);
  \draw[thick arrow] (L2 1) -- (L3 1);

  \draw[thin arrow] (L3 0) -- (L4 0);
  \draw[thick arrow] (L3 1) -- (L4 1);
  \draw[thin arrow] (L4 0) -- (L5 0);
  \draw[thick arrow] (L4 1) -- (L5 1);
  
  \draw[thick arrow] (L5 0) -- (L6 0);
  \draw[thin arrow] (L5 1) -- (L6 0);
  \draw[thick arrow] (L5 1) -- (L6 1);

  \draw[thick arrow] (L6 0) -- (L7);
  \draw[thin arrow] (L6 1) -- (L7);

  \draw[thin arrow] (L7) -- (L8);
\end{tikzpicture}}}
\end{minipage}
\caption{$B_2$ crystals $B(\Lambda_1+\Lambda_2),B(3\Lambda_1),B(2\Lambda_2)$ from left to right}
\label{crystalgraphs}
\end{figure}

\subsection{Variants of axioms}
By the symmetry with respect to $\KE{k}$ and $\KF{k}$ (resp. $\varepsilon_k$ and $\varphi_k$) for $k\in\{i,j\}$ and $i\ne j\in I$ with $A|_{i,j}=B_2$, 
(S8') holds (see 
Proposition \ref{bsym}).
\begin{enumerate}
\item[(S8')] \mbox{$\forall x\in X,\KE{i}x\ne\ZERO\ne\KE{j}x,\Delta(x)=(1,1),\varepsilon_i(x)\geq 2\Rightarrow$ (C$^-_1$).}
\end{enumerate}
Here, (C$^-_1$) is given as $\exists z=\KE{i}\KE{j}^2\KE{i}^2x=\KE{j}\KE{i}^3\KE{j}x$.
Similarly, the full symmetric version of (S6) (instead of (S7)) and the symmetric version of (S9) hold for $i\ne j\in I$ with $A|_{i,j}=B_2$, but we do not need them.

We can replace (P$^-_1$), (Q$^-_1$) with (P$^-$), (Q$^-$) respectively (and independently).

\hspace{0mm}

\noindent{(P$^-$)} 
\mbox{$(\Delta^f_{\varphi}(i,j,y),\Delta^f_{\varphi}(i,j,y'))=(1,1)\Rightarrow y'=\KE{i}\KE{j}\KE{i}\KE{j}\KE{i}x=\KE{j}\KE{i}^3\KE{j}x,\Delta^f_{\varphi}(j,i,y')=1$}.

\noindent{(Q$^-$)} 
\mbox{$(\Delta^f_{\varphi}(i,j,y),\Delta^f_{\varphi}(i,j,y'))=(0,1)$}

\quad\quad\mbox{$\Rightarrow\exists z=\KE{j}\KE{i}^2\KE{j}\KE{i}\KE{j}\KE{i}x=\KE{j}\KE{i}^3\KE{j}^2\KE{i}x=\KE{i}\KE{j}^2\KE{i}^3\KE{j}x=\KE{i}\KE{j}\KE{i}\KE{j}\KE{i}^2\KE{j}x,\Delta'(z)=(1,2)$}.

\hspace{0mm}

A reason why the shorter version operates well is that 
a proposition (precisely, Proposition \ref{confprop}) that is used in the proof of Theorem \ref{maintheorem}
just needs weak Church-Rosser property (a.k.a. local confluence property, see ~\cite[\S2.7]{BN}).
\begin{Def}\label{localhomogconf}
Let $X$ be a good $I$-colored directed graph.
We say $X$ has a \emph{homogeneous local confluence} property if (the last equality is one as multisets)
\begin{align*}
{} &\forall x\in X,\forall i\ne\forall j\in I,\KE{i}x\ne\ZERO\ne\KE{j}x,\\
&\quad\Rightarrow 
\exists s\geq2,\exists (i_1,\cdots,i_s),\exists (i'_1,\cdots,i'_s)\in I^s,\\
&\quad\quad\quad
i_s=i,i'_s=j,\exists z=
\KE{i_1}\cdots\KE{i_s}x=\KE{i'_1}\cdots\KE{i'_s}x,
\{i_k\mid 1\leq k\leq s\}=\{i'_k\mid 1\leq k\leq s\}.
\end{align*}
\end{Def}

We choose (P$^-_1$) (resp. (Q$^-_1$)) by just a reason that
they involve minimum number of $i,j$-strings in the confluence relation in (P$^-$) (resp. (Q$^-$)) because of the maximum powers on Kashiwara operators.
However by similar reasons mentioned above involving Definition \ref{localhomogconf},
(P$^-_1$) (resp. (Q$^-_1$)) has some alternatives other than (P$^-$) (resp. (Q$^-$)). 
We will not discuss such alternatives because (P$^-_1$) (resp. (Q$^-_1$)) seem to be nice.
For example, if the extra implication in Remark \ref{luck} does not hold, 
we need to add $(\Delta^e_{\varepsilon}(i,j,\KF{i}^2\KF{j}z),\Delta^e_{\varepsilon}(i,j,\KF{i}^2\KF{j}^2\KF{i}z))=(0,1)$
to (Q$_1^-$) so that the assumption $(\Delta^e_{\varepsilon}(i,j,y),\Delta^e_{\varepsilon}(i,j,y'))=(0,1)$ in (D$^+$)
is certainly satisfied when we substitute $z$ to $x$ in (S7). Yet luckily, we need not.

\begin{Rem}\label{luck}
In (Q$^-_1$), the symmetry $\KF{i}^2\KF{j}z=\KF{i}y',\KF{i}^2\KF{j}^2\KF{i}z=\KF{i}y$ (see Figure \ref{crystalgraphs})
and Remark \ref{henkanef}
imply $(\Delta^e_{\varepsilon}(i,j,\KF{i}^2\KF{j}z),\Delta^e_{\varepsilon}(i,j,\KF{i}^2\KF{j}^2\KF{i}z))=(0,1)$ because
\begin{align*}
(\Delta^e_{\varepsilon}(i,j,\KF{i}y'),\Delta^e_{\varepsilon}(i,j,\KF{i}y))
&= -(\Delta^f_{\varepsilon}(i,j,y'),\Delta^f_{\varepsilon}(i,j,y)) \\
&= -(a_{ji}+\Delta^f_{\varphi}(i,j,y'),a_{ji}+\Delta^f_{\varphi}(i,j,y))
=(0,1).
\end{align*}
\end{Rem}

\subsection{Comparison with previous studies}\label{prevstu}
Finding a set of local axioms that characterizes $B_2$ regular crystals has been a well-known
open problem since ~\cite{Ste}.
In rest of this subsection, we compare Theorem \ref{maintheorem} with the previous studies.

\subsubsection{Comparison with ~\cite{Ster}}
The purpose of ~\cite{Ster} was to prove that the crystal graph satisfies
the confluence relations in (P$^-$),(Q$^-$) (and (R$^-$) that implies $\KE{i}\KE{j}^2\KE{i}x=\KE{j}\KE{i}^2\KE{j}x$ by (A$^+_{i,j}$) in (S4)) which were 
observed in ~\cite[p.4822]{Ste} and ~\cite[Remark 6.7.(d)]{Ste2} (see ~\cite[Theorem 1]{Ster}).
As he wrote in Introduction of ~\cite{Ster}, it did not provide a local characterization of
the crystal graph.
To determine which confluence relation actually occurs for $x$ with $\Delta(x)=(1,2)$ from the local structure near $x$,
existences of $y$ and $y'$ in (D$^-$) are crucial. 

\begin{Rem}\label{localaxiom}
In this paper, ``local condition'' for $x\in X$ is an axiom 
that involves only $\Delta^g_{\beta}(k,\ell,y),\beta_k(y)$ and $=$ between $y$'s, where $k,\ell\in I,g\in\{e,f\},\beta\in\{\varepsilon,\varphi\}$ and $y$ is ``near'' $x$.
It means that we can go back and forth between $x$ and $y$ at most $N$ arrows, where $N$ is a constant.
In Stembridge's axiom $N=4$ and in ours $N=7$.
Note that the assumption of the existence of a (unique) maximum element
in Theorem \ref{StembridgeCrystal} and Theorem \ref{maintheorem} is not a local condition.
\end{Rem}

Other missing axioms (first singled out in this paper) play the following role.
\begin{itemize}
\item[(S8)] compensates the symmetry breaking in (P$^-_1$) 
in that $\Delta'(z)=(1,1)$ instead of $\Delta'(z)=(1,2)$, where \mbox{$z=\KE{i}^2\KE{j}^2\KE{i}x=\KE{j}\KE{i}^3\KE{j}x(=y')$},
\item[(S9)] handles the fact $\KF{i}^2\KF{j}^2\KF{i}z$ is ``under'' or ``below'' $x$ in (R$^-$), where $z=\KE{i}\KE{j}^2\KE{i}x=\KE{j}\KE{i}^2\KE{j}x$
notwithstanding $\Delta'(z)=(1,2)$ (using the notation that will be introduced in \S\ref{uniqconf}, what we mean is $\DIST(\KF{i}^2\KF{j}^2\KF{i}z)>\DIST(x)$).
\end{itemize}

\begin{Rem}\label{iterative}
As ~\cite[Remark 1.5]{Ste}, Theorem \ref{StembridgeCrystal2} gives an iterative algorithm
that draws simply-laced highest weight crystals (the proof of ~\cite[Proposition 1.4]{Ste} provides an algorithm). 
Especially thanks to (S9),
this straightforwardly translates to Theorem \ref{maintheorem} (the proof of Proposition \ref{uniqprop} provides an algorithm). 
\end{Rem}

\subsubsection{Comparison with ~\cite{DKK}}
In ~\cite{DKK}, they gave a set of axioms (B0),(B1),(B2),(B3),(B3'),(B5)--(B13),(B5')--(B13') and claimed 
that it characterizes $B_2$ regular crystals (see the first paragraph of ~\cite[\S3]{DKK}). 
The idea is different from ~\cite{Ste} while this paper is a small modification of ~\cite{Ste} as in Remarks \ref{localaxiom} and \ref{iterative}.
For example, it is not clear how the axioms of ~\cite{DKK} are translated to an iterative algorithm mentioned in Remark \ref{iterative}.
It seems to be interesting to clarify an relationship between the idea of ~\cite{DKK} and ~\cite{Ste} (and ours). 
We emphasize that Theorem \ref{maintheorem} is not an improvement of ~\cite{DKK} in that 
there is no logical dependencies on ~\cite{DKK} (and there also seems no
obvious implications of our results to ~\cite{DKK}).

The author can say that the axioms of ~\cite{DKK} are \emph{not local} in the sense of Remark \ref{localaxiom}.
For example, the axiom (B1) (see ~\cite[p.272]{DKK}) 
requires to check that every 2-string\footnote{We adapt the convention of $B_2$ to ours.
Our index $i\in\{1,2\}$ of $B_2$ is $3-i$ in ~\cite{DKK}.} 
\emph{has exactly one} central edge or central vertex. 
Though they give a local criterion for checking \emph{an edge} or \emph{a vertex} to be central\footnote{For 
a newcomer on ~\cite{DKK}, we remark that central does not mean to be located at the middle position 
in a 2-string literally.}
in ~\cite[p.273,lines 21--29]{DKK} (which is mathematically wrong as we will point out in Remark \ref{dkkerror}), 
the axiom (B1), which is a constraint for \emph{a string}, is not locally checkable in the sense of Remark \ref{localaxiom} (at least,
one should say that a local checkability of (B1) is not trivial).

\begin{Rem}\label{dkkerror}
The criterion on central vertex in ~\cite[p.273,(i),(ii),(iii)]{DKK} is wrong.
To see this, consider a 2-string from $\KF{i}^2\KF{j}z$ to $w$ in the left graph in Figure \ref{crystalgraphs}. 
The edge $(\KF{i}^2\KF{j}z,w)$ is not central by ~\cite[p.273,line 26]{DKK}.
Both the vertices $\KF{i}^2\KF{j}z$ and $w$ are not central because they satisfy neither of ~\cite[p.273,(i),(ii),(iii)]{DKK}.
Then, this 2-string does not satisfy (B1) contradicting the claim that  
the axioms of ~\cite{DKK} characterize $B_2$ regular crystals.
\end{Rem}

\begin{Rem}\label{dkkerror2}
~\cite[Corollary 2]{DKK} and ~\cite[Corollary 3]{DKK} that are cited in the proof of 
~\cite[Theorem 4]{DKK} are wrong. In ~\cite[\S3]{DKK}, they proved both
corollaries assuming only axioms (B0)--(B3).
Though we omit a step-by-step verification as in Remark \ref{dkkerror},
it is checked that 
\begin{enumerate}
\item the figure below, where we adapt the notation in ~\cite{DKK}\footnote{As 
in the first paragraph of ~\cite[\S2]{DKK},
vertical arrows are 2-edges and horizontal arrows are 1-edges.
A decoration by $\blacklozenge$ indicates that the decorated 1-edge is central.
A central vertex is depicted by $\bullet$.
Thus, in the figure, 
1-edges $(u,v),(w,x)$ are central edges and vertices $u',y$ are central vertices.}, satisfies (B0)--(B3).
\item ~\cite[Corollary 2]{DKK} does not hold for this figure in that $x$ is not central,
\item ~\cite[Corollary 3]{DKK} does not hold for this figure in that $x$ is not central.
\end{enumerate}
\end{Rem}

\begin{center}
\scalebox{0.6}{
\begin{tikzpicture}[
  usual arrow/.style = {
    ->, thick
  },
  ]
  \begin{scope}[xshift=9cm]

  \node[draw, circle, inner sep = 2pt,fill=black, label = {[left]\huge{$u'$}}] (L0) {};

  \node[draw, circle, inner sep = 2pt,above = of L0, label={[left]\huge{$u$}}] (L1){};
  \node[draw, circle, inner sep = 2pt,right = of L1, label={[right]\huge{$v$}}] (L2){};

  \node[draw, circle, inner sep = 2pt,above = of L1, label={[left]\huge{$w$}}] (L3){};
  \node[draw, circle, inner sep = 2pt,above = of L2, label={[right]\huge{$x$}}] (L4){};
  
  \node[draw, circle, inner sep = 2pt,fill=black, above = of L4, label={[right]\huge{$y$}}] (L5){};

  \draw[usual arrow] (L0) -- (L1);
  \draw[usual arrow] (L1) to node {$\blacklozenge$} (L2);
  \draw[usual arrow] (L1) -- (L3);
  \draw[usual arrow] (L2) -- (L4);
  \draw[usual arrow] (L4) -- (L5);
  \draw[usual arrow] (L3) to node {$\blacklozenge$} (L4);
\end{scope}
\end{tikzpicture}}
\end{center}

\subsection{Possible applications and future directions}\label{nextsection}
Our motivation for Theorem \ref{maintheorem} comes from a joint work with Masaki Watanabe
on a generalization of Schur partition theorem~\cite[Theorem 1.2]{TW}.
In ~\cite[Theorem 1.6]{TW}, we give a human readable proof for ~\cite[Conjecture 2]{An4} originally proved by computer in ~\cite{ABO}.
The strategy is to prove that a certain subset $S_5$ of the set of all partitions (singled out by Andrews when $p=5$,
and then generalized for any odd $p\geq 3$ in ~\cite[Definition 1.1]{TW})
is an $A^{(2)}_{4}$ regular crystal.
Though there is a concrete, combinatorial definition of Kashiwara operators $\KE{i},\KF{i}$ on $S_5$,
establishing the regularity can be achieved only through the theory of perfect crystals~\cite{KKMMNN,KKMMNN2}.
This method of proof 
is indirect and shares a spirit of categorification that can be called 
``an interpretation of a rich mathematical structure''
like the original proof of positivity of Kazhdan-Lusztig polynomials.

Theorem \ref{maintheorem} may give a direct way to establish the regularity 
and is expected to give applications for the theory of partitions.
The author hopes Theorem \ref{maintheorem} gives applications
in the study of Kirillov-Reshetikhin crystals (see ~\cite{FOS} and the references therein)
and in new directions of research such as ~\cite{Lyn}. 

Subsequent to Theorem \ref{maintheorem}, it is natural to pursue local axioms of $G_2$ regular crystals.
It seems possible once a correct axiom is found (quite likely through computer experimentation).
However it seems to be not a short list (see the last paragraph of ~\cite{Ster}) and 
it is not unreasonable to foresee that we invoke computer verifications in a proof.
An advantage of Theorem \ref{maintheorem} is that
the axioms are not so complicated compared with Stembridge $A$-regularity axioms.
The author also believes that Theorem \ref{maintheorem} is often enough in practice since it is applicable 
to finite GCMs except $G_2$ and 
affine GCMs except $A^{(1)}_{1},G^{(1)}_{2},A^{(2)}_{2},D^{(3)}_4$.

Finally, the author has an impression that there exists no \emph{finite} complete list of local confluence relations (see Definition \ref{localhomogconf}) for $A^{(1)}_{1},A^{(2)}_{2}$  regular crystals.
Even if it is the case, at least logically, there is a possibility that we can``write down \emph{uniformly}'' 
(like Jordan normal forms in the classification of indecomposable $\mathbb{C}[x]$-modules)
infinitely many local confluence relations. 
Finding that or proving an impossibility of that under a reasonable formulation (like wild representation type in the classification of indecomposable $\mathbb{C}\langle x,y\rangle$-modules) might 
give an insight for regular crystals of affine types.

We hope that this paper will push Stembridge's local approach further.


\noindent{\bf Acknowledgments.} The author thanks Hironori Oya, Naoki Fujita and Travis Scrimshaw for helpful discussions.
He is also grateful to Motoki Takigiku for 
teaching me the example in Remark \ref{dkkerror2} and
to anonymous referees for many kind comments 
that led to considerable improvement of the presentation.
He was supported by JSPS Kakenhi Grant 17K14154.




\section{Proof of Theorem \ref{maintheorem} : $B(\lambda)$ satisfies the axioms in Theorem \ref{maintheorem}}\label{kakunin}
\subsection{A reduction to $A=B_2$}\label{kichaku}
Recall 
the decomposition \mbox{(see ~\cite[Proposition 4.3]{Kas3})}
\begin{align*}
B(\lambda)=\bigsqcup_{b}C(b),\quad
C(b):=\{\KF{p_1}\cdots\KF{p_s}b\mid s\geq 0,(p_1,\cdots,p_s)\in\{i,j\}^s\}\setminus\{\ZERO\}
\end{align*}
for $i\ne j\in I$, where $b$ runs through all $b\in B(\lambda)$ with $\varepsilon_{i}(b)=\varepsilon_{j}(b)=0$
and we identify $\KF{i}:B\to B\sqcup\{\ZERO\}$ with $\KF{i}:B\sqcup\{\ZERO\}\to B\sqcup\{\ZERO\}$ by $\KF{i}\ZERO=\ZERO$.
As $A|_{i,j}$-crystals $C(b)\cong B(\lambda_b)$
with $\langle h_i,\lambda_b\rangle=\varphi_i(b),\langle h_j,\lambda_b\rangle=\varphi_j(b)$.

Combined with Theorem \ref{StembridgeCrystal}, to prove that $B(\lambda)$ satisfies the axioms in Theorem \ref{maintheorem},
it is enough to prove that $B_2$ highest weight crystals satisfy (S6),(S7),(S8),(S9) putting $i=1,j=2$.

In the rest of \S\ref{kakunin}, 
we assume $A=B_2=\begin{pmatrix} 2 & -2 \\ -1 & 2 \end{pmatrix}$ (indexed by $I=\{1,2\}$) 
and prove Proposition \ref{kakunin1}, Proposition \ref{kakunin2}, Proposition \ref{kakunin3} 
in \S\ref{pkakunin1}, \S\ref{pkakunin2}, \S\ref{pkakunin3}
that imply ((S6),(S7)),(S8),(S9) respectively thanks to Proposition \ref{bsym}.
In \S\ref{pkakunin1}, there might be some overlap with calculations by ~\cite{Ster} via Kashiwara-Nakashima tableaux~\cite{KN}.
Our proof is based on Lusztig parameterization~\cite{Lus}.

\begin{Prop}\label{kakunin1}
Fix $\lambda\in P^+$ and take $x\in B(\lambda)$.
If $\KE{1}x\ne\ZERO\ne\KE{2}x$ and $(\Delta^e_{\varepsilon}(1,2,x),\Delta^e_{\varepsilon}(2,1,x))=(1,2)$, then 
$\exists y'=\KE{1}^2\KE{2}^2\KE{1}x$ and we have exactly (i.e., exclusively) one of the following 3 cases.
Here $\Delta'=(\Delta^f_{\varphi}(1,2,z),\Delta^f_{\varphi}(2,1,z))$ and $\Delta''=(\Delta^f_{\varphi}(1,2,y),\Delta^f_{\varphi}(1,2,y'))$, $y=\KE{1}^2\KE{2}x$.
\end{Prop}

\noindent{(case $\Delta''=(1,1)$)} 
\mbox{$y'=\KE{1}\KE{2}\KE{1}\KE{2}\KE{1}x=\KE{2}\KE{1}^3\KE{2}x,\Delta^f_{\varphi}(2,1,y')=1$}.

\noindent{(case $\Delta''=(0,1)$)}
\mbox{$\exists z=\KE{2}\KE{1}^2\KE{2}\KE{1}\KE{2}\KE{1}x=\KE{2}\KE{1}^3\KE{2}^2\KE{1}x=\KE{1}\KE{2}^2\KE{1}^3\KE{2}x=\KE{1}\KE{2}\KE{1}\KE{2}\KE{1}^2\KE{2}x$,$\Delta'=(1,2)$}.

\noindent{(case $\Delta''=(0,0)$)}
\mbox{$\KF{2}y'=\KE{1}y,\Delta^f_{\varphi}(2,1,y')=2,\Delta^f_{\varphi}(2,1,\KF{1}^2y')=0$}.



\begin{Prop}\label{kakunin2}
Fix $\lambda\in P^+$ and take $x\in B(\lambda)$.
If $\KE{1}x\ne\ZERO\ne\KE{2}x$ and $\varepsilon_1(x)\geq 2$,$(\Delta^e_{\varepsilon}(1,2,x),\Delta^e_{\varepsilon}(2,1,x))=(1,1)$, then $\exists z=\KE{1}\KE{2}^2\KE{1}^2x=\KE{1}\KE{2}\KE{1}\KE{2}\KE{1}x=\KE{2}\KE{1}^3\KE{2}x$.
\end{Prop}

\begin{Prop}\label{kakunin3}
Fix $\lambda\in P^+$ and take $x\in B(\lambda)$.
If $\KE{1}x\ne\ZERO\ne\KE{2}x$ and $(\Delta^e_{\varepsilon}(1,2,x),\Delta^e_{\varepsilon}(2,1,x))=(0,2)$,
$\KE{2}\KE{1}^2x\ne\ZERO,\Delta^e_{\varepsilon}(2,1,\KE{1}^2x)=0$, then $\exists z=\KE{2}\KE{1}^3\KE{2}z=\KE{2}\KE{1}^2\KE{2}\KE{1}x=\KE{1}\KE{2}^2\KE{1}^2x$.
\end{Prop}


The following is a version of the Lusztig involution (see ~\cite[Proposition 21.1.2]{Lu2} and ~\cite[\S7.4]{Kas3}).
We remark that it is not trivial in that it does not hold for any $A$ (e.g. when $A=A_2$).
In our case of $A=B_2$, it holds by $-w_0\lambda=\lambda$ for $\lambda\in P^+$.

\begin{Prop}\label{bsym}
For $\lambda\in P^+$, there is an involution $w:B(\lambda)\ISOM B(\lambda)$ such that
\begin{enumerate}
\item $\forall b\in B(\lambda), \forall i\in I, \varepsilon_i(b)=\varphi_i(w(b))$,
\item $\forall b\in B(\lambda), \forall i\in I, \KE{i}b\ne\ZERO\Rightarrow w(\KE{i}b)=\KF{i}(w(b))$.
\end{enumerate}
\end{Prop}

\subsection{A realization of $B_2$ highest weight crystals}\label{rb2}
The choice $\boldsymbol{i}=s_1s_2s_1s_2$ (resp. $\boldsymbol{j}=s_2s_1s_2s_1$) of a reduced expression of $w_0$ gives 
the convex order on the positive roots $\Delta^+=\{\alpha_1,\alpha_2,\alpha_1+\alpha_2,2\alpha_1+\alpha_2\}$ as
$\alpha_1<2\alpha_1+\alpha_2<\alpha_1+\alpha_2<\alpha_2$ (resp. $\alpha_2<\alpha_1+\alpha_2<2\alpha_1+\alpha_2<\alpha_1$).
Lusztig's PBW parameterization (see ~\cite[\S3]{BZ}) associated with $\boldsymbol{k}\in\{\boldsymbol{i},\boldsymbol{j}\}$
gives a realization of $B(\infty)$ on $\mathbb{N}^4$, where $4=\ell(w_0)$.
Though each carrying the same information of the other in principle, we prefer considering both simultaneously:
\begin{align*}
B(\infty) &= \{
(\boldsymbol{a},\boldsymbol{x})\in\mathbb{N}^4\times\mathbb{N}^{4}\mid \LusR(\boldsymbol{a})=\boldsymbol{x}
\},\\
\WTTT(\boldsymbol{a},\boldsymbol{x}) &= -(x_2+2x_3+x_4)\alpha_1-(x_1+x_2+x_3)\alpha_2, \\
\varepsilon_1(\boldsymbol{a},\boldsymbol{x}) &= a_1,\quad
\varepsilon_2(\boldsymbol{a},\boldsymbol{x}) = x_1,\quad
\varphi_i(\boldsymbol{a},\boldsymbol{x}) = \varepsilon_i(\boldsymbol{a},\boldsymbol{x}) + \langle h_i,\WTTT(\boldsymbol{a},\boldsymbol{x})\rangle, \\
\KE{1}(\boldsymbol{a},\boldsymbol{x}) &=
\begin{cases}
((a_1-1,a_2,a_3,a_4),\LusR(a_1-1,a_2,a_3,a_4)) & (\varepsilon_1(\boldsymbol{a},\boldsymbol{x})>0) \\
\ZERO                                           & (\varepsilon_1(\boldsymbol{a},\boldsymbol{x})=0),
\end{cases}\\
\KE{2}(\boldsymbol{a},\boldsymbol{x}) &=
\begin{cases}
(\LusRR(x_1-1,x_2,x_3,x_4),(x_1-1,x_2,x_3,x_4)) & (\varepsilon_2(\boldsymbol{a},\boldsymbol{x})>0) \\
\ZERO                                           & (\varepsilon_2(\boldsymbol{a},\boldsymbol{x})=0), 
\end{cases}\\
\KF{1}(\boldsymbol{a},\boldsymbol{x}) &= ((a_1+1,a_2,a_3,a_4),\LusR(a_1+1,a_2,a_3,a_4)), \\
\KF{2}(\boldsymbol{a},\boldsymbol{x}) &= (\LusRR(x_1+1,x_2,x_3,x_4),(x_1+1,x_2,x_3,x_4)).
\end{align*}
Here $\boldsymbol{a}=(a_1,a_2,a_3,a_4),\boldsymbol{x}=(x_1,x_2,x_3,x_4)\in\mathbb{N}^4$ and 
note that $\WTTT(\boldsymbol{a},\boldsymbol{x})$ can also be expressed as $\WTTT(\boldsymbol{a},\boldsymbol{x})=-(a_1+2a_2+a_3)\alpha_1-(a_2+a_3+a_4)\alpha_2$.
The function $\LusR$ switches PBW parameterizations (see ~\cite[\S3]{Gra} and ~\cite[\S1.9]{Lu3}).
\begin{Def}\label{lusrr}
Let $\LusR:\mathbb{N}^4\to\mathbb{N}^4,(a,b,c,d)\mapsto(n_1,\mu-n_2,n_2+n_3-\mu,n_4-2n_3+\mu)$ be a bijection 
with $\LusRR:\mathbb{N}^4\to\mathbb{N}^4,(a,b,c,d)\mapsto(p_1,\nu-p_2,2p_2+p_3-2\nu,p_4-p_3+\nu)$.
\begin{align*}
\begin{array}{lcl}
n_1=\max(b,\max(b,d)+c-a) & {} & p_1=\max(b,\max(b,d)+2(c-a)) \\
n_2=\max(a,c)+2b          & {} & p_2=\max(a,c)+b              \\
n_3=\min(c+d,a+\min(b,d)) & {} & p_3=\min(2c+d,2a+\min(b,d))  \\
n_4=\min(a,c)             & {} & p_4=\min(a,c)                \\
\mu=\max(2n_3,n_2+n_4)    & {} & \nu=\max(p_3,p_2+p_4)
\end{array}
\end{align*}
\end{Def}

In summary, $(B(\infty),\WTTT,({\KE{i}})_{i\in I},({\KF{i}})_{i\in I},(\varepsilon_i)_{i\in I},(\varphi_i)_{i\in I})$ is a realization of the crystal $B(\infty)$. 
Thanks to ~\cite[Proposition 8.2]{Kas3}, $B(\lambda)$ is isomorphic to 
\begin{align*}
\{b\otimes t_{\lambda}\mid b\in B(\infty),\forall i\in I,\varepsilon^{\ast}_i(b)\leq\langle h_i,\lambda\rangle\}
\end{align*}
in $B(\infty)\otimes T_{\lambda}$ by the tensor product ~\cite[\S7.3]{Kas3}, where $T_{\lambda}$ is given by $T_{\lambda}=\{t_{\lambda}\},\WTTT(t_{\lambda})=\lambda,\varphi_i(t_{\lambda})=\varepsilon_i(t_{\lambda})=-\infty,\KE{i}t_{\lambda}=\KF{i}t_{\lambda}=\ZERO$ (see ~\cite[Example 7.3]{Kas3}).
Though we do not explain the $\ast$-structure (see ~\cite[\S8.3]{Kas3}), we use the fact $\varepsilon^{\ast}_1(\boldsymbol{a},\boldsymbol{x})=x_4,\varepsilon^{\ast}_2(\boldsymbol{a},\boldsymbol{x})=a_4$ (see ~\cite[Proposition 3.3.(iii)]{BZ}, ~\cite[\S2.11]{Lus}). Therefore:
\begin{Prop}\label{regcry}
For $\lambda\in P^+$, $B(\lambda)$ is realized as $(B(\lambda),\WTTT,({\KE{i}})_{i\in I},({\KF{i}})_{i\in I},(\varepsilon_i)_{i\in I},(\varphi_i)_{i\in I})$. 
\begin{align*}
B(\lambda) &= \{
(\boldsymbol{a},\boldsymbol{x})\in\mathbb{N}^4\times\mathbb{N}^{4}\mid \LusR(\boldsymbol{a})=\boldsymbol{x},x_4\leq\langle h_1,\lambda\rangle,a_4\leq\langle h_2,\lambda\rangle
\},\\
\WTTT(\boldsymbol{a},\boldsymbol{x}) &= \lambda-(x_2+2x_3+x_4)\alpha_1-(x_1+x_2+x_3)\alpha_2, \\
\varepsilon_1(\boldsymbol{a},\boldsymbol{x}) &= a_1,\quad
\varepsilon_2(\boldsymbol{a},\boldsymbol{x}) = x_1,\quad
\varphi_i(\boldsymbol{a},\boldsymbol{x}) = \varepsilon_i(\boldsymbol{a},\boldsymbol{x}) + \langle h_i,\WTTT(\boldsymbol{a},\boldsymbol{x})\rangle, \\
\KE{1}(\boldsymbol{a},\boldsymbol{x}) &=
\begin{cases}
((a_1-1,a_2,a_3,a_4),\LusR(a_1-1,a_2,a_3,a_4)) & (\varepsilon_1(\boldsymbol{a},\boldsymbol{x})>0) \\
\ZERO                                           & (\varepsilon_1(\boldsymbol{a},\boldsymbol{x})=0),
\end{cases}\\
\KE{2}(\boldsymbol{a},\boldsymbol{x}) &=
\begin{cases}
(\LusRR(x_1-1,x_2,x_3,x_4),(x_1-1,x_2,x_3,x_4)) & (\varepsilon_2(\boldsymbol{a},\boldsymbol{x})>0) \\
\ZERO                                           & (\varepsilon_2(\boldsymbol{a},\boldsymbol{x})=0), 
\end{cases}\\
\KF{1}(\boldsymbol{a},\boldsymbol{x}) &=
\begin{cases}
((a_1+1,a_2,a_3,a_4),\LusR(a_1+1,a_2,a_3,a_4)) & (\varphi_1(\boldsymbol{a},\boldsymbol{x})>0) \\
\ZERO                                           & (\varphi_1(\boldsymbol{a},\boldsymbol{x})=0),
\end{cases}\\
\KF{2}(\boldsymbol{a},\boldsymbol{x}) &=
\begin{cases}
(\LusRR(x_1+1,x_2,x_3,x_4),(x_1+1,x_2,x_3,x_4)) & (\varphi_2(\boldsymbol{a},\boldsymbol{x})>0) \\
\ZERO                                           & (\varphi_2(\boldsymbol{a},\boldsymbol{x})=0).
\end{cases}
\end{align*}
\end{Prop}

\subsection{Auxiliary formulas}
We list formulas that are verified by direct calculation.


\begin{Lem}\label{lessthanorequalto3}
For $\boldsymbol{a}=(a_1,a_2,a_3,a_4)\in\mathbb{N}^4$ with $a_3\geq a_1$, $\LusR(\boldsymbol{a})$ is given by
\begin{align*}
(\max(a_2,a_4)+a_3-a_1,a_1,\min(a_2,a_4),a_3+2a_2-2\min(a_2,a_4)).
\end{align*}
\end{Lem}

\begin{Cor}\label{greaterthanorequaltosub1}
For $\lambda\in P^+$, take $m=((a_1,a_2,a_3,a_4),(x_1,x_2,x_3,x_4))\in B(\lambda)$.
If $a_3\geq a_1\geq 1$ and $x_1\geq 1$, then $\Delta^e_{\varepsilon}(2,1,m)=\max(0,2+a_1-a_3+2a_2-2\max(a_2,a_4))$.
\end{Cor}

\begin{Lem}\label{lessthanorequalto4}
For $\boldsymbol{x}=(x_1,x_2,x_3,x_4)\in\mathbb{N}^4$ with $x_3\geq x_1$, $\LusRR(\boldsymbol{x})$ is given by
\begin{align*}
(\max(x_2,x_4)+2(x_3-x_1),x_1,\min(x_2,x_4),x_3+x_2-\min(x_2,x_4)).
\end{align*}
\end{Lem}

\begin{Cor}\label{greaterthanorequaltosub2}
For $\lambda\in P^+$, take $m=((a_1,a_2,a_3,a_4),(x_1,x_2,x_3,x_4))\in B(\lambda)$.
If $x_3\geq x_1\geq 1$ and $a_1\geq 1$, then $\Delta^e_{\varepsilon}(1,2,m)=\max(0,1+x_1-x_3+x_2-\max(x_2,x_4))$.
\end{Cor}

\begin{Lem}\label{lessthanorequalto1}
For $\boldsymbol{a}=(a_1,a_2,a_3,a_4)\in\mathbb{N}^4$ with $a_3\leq a_1$, $\LusR(\boldsymbol{a})$ is given by
\begin{align*}
\begin{cases}
(a_2,a_3,a_4,a_1+2a_2-2a_4) & (a_2\geq a_4+(a_3-a_1)/2) \\
(a_2,2a_3+2a_4-a_1-2a_2,a_1+2a_2-(a_3+a_4),a_3) & (a_4+a_3-a_1\leq a_2\leq a_4+(a_3-a_1)/2) \\
(a_4+a_3-a_1,a_1,a_2,a_3) & (a_2\leq a_4+a_3-a_1).
\end{cases}
\end{align*}
\end{Lem}

\begin{Lem}\label{lessthanorequalto2}
For $\boldsymbol{x}=(x_1,x_2,x_3,x_4)\in\mathbb{N}^4$ with $x_3\leq x_1$, $\LusRR(\boldsymbol{x})$ is \mbox{given by}
\begin{align*}
\begin{cases}
(x_2,x_3,x_4,x_1+x_2-x_4) & (x_2\geq x_4+x_3-x_1) \\
(x_2,2x_3+x_4-x_1-x_2,2x_1+2x_2-2x_3-x_4,x_3) & (x_4+2(x_3-x_1)\leq x_2\leq x_4+x_3-x_1) \\
(x_4+2(x_3-x_1),x_1,x_2,x_3) & (x_2\leq x_4+2(x_3-x_1)).
\end{cases}
\end{align*}
\end{Lem}

\begin{Cor}\label{greater2}
For $\lambda\in P^+$, take $m=((a_1,a_2,a_3,a_4),(x_1,x_2,x_3,x_4))\in B(\lambda)$.
If $a_1>a_3$ and $x_1>x_3$, then $\Delta^e_{\varepsilon}(1,2,m)\Delta^e_{\varepsilon}(2,1,m)=0$.
\end{Cor}

\begin{proof}
By Lemma \ref{lessthanorequalto1}, 
$x_1>x_3$ only happens when $a_2\geq a_4+(a_3-a_1)/2$ or $a_2\leq a_4+a_3-a_1$.
In the former case, $a_2\geq a_4+(a_3-(a_1-1))/2$ also holds because $a_2=x_1>x_3=a_4$ and $a_1>a_3$.
This implies $\Delta^e_{\varepsilon}(1,2,m)=a_2-a_2=0$. In the latter case, we see 
$x_4+2(x_3-(x_1-1))-x_2=2+(a_3-a_1)+2(x_3-x_1)<0$ similarly.
By Lemma \ref{lessthanorequalto2}, $\Delta^e_{\varepsilon}(2,1,m)=a_1-x_2=0$.
\end{proof}

\subsection{Proof of Proposition \ref{kakunin1}}\label{pkakunin1}
First, we show
\begin{align*}
\{m\in B(\lambda) &\mid \varepsilon_1(m),\varepsilon_2(m)>0,
(\Delta^e_{\varepsilon}(1,2,m),\Delta^e_{\varepsilon}(2,1,m))=(1,2)\}=X_1\sqcup X_2\sqcup X_3,
\end{align*}
where $X_1,X_2,X_3$ are defined as follows.
\begin{align*}
X_1 &:= \{((a,b,a,b),(b,a,b,a))\mid a,b\geq 1\}\cap B(\lambda), \\
X_2 &:= \{((a,b,a,c),(b,a,c,a+2b-2c))\mid a\geq 1,0\leq c<b\}\cap B(\lambda), \\
X_3 &:= \{((a,b,c,a+b-c),(b,a,b,c))\mid b\geq 1,0\leq c<a\}\cap B(\lambda).
\end{align*}

Since the inclusion $\supseteq$ is verified by direct calculation, we show $m\in X_1\sqcup X_2\sqcup X_3$ for any $m=((a_1,a_2,a_3,a_4),(x_1,x_2,x_3,x_4))\in B(\lambda)$ such that $a_1,x_1>0,(\Delta^e_{\varepsilon}(1,2,m),\Delta^e_{\varepsilon}(2,1,m))=(1,2)$. By Corollaries \ref{greaterthanorequaltosub1} and \ref{greaterthanorequaltosub2}, we have $a_1\geq a_3,x_1\geq x_3$
and thus we get $a_1=a_3$ or $x_1=x_3$ by Corollary \ref{greater2}.
By Corollaries \ref{greaterthanorequaltosub1} and \ref{greaterthanorequaltosub2}, 
this implies $a_2\geq a_4$ (i.e., $m\in X_1\sqcup X_2$) 
or $x_2\geq x_4$ (i.e., $m\in X_1\sqcup X_3$) respectively.

\subsubsection{}
For $x=((a,b,a,b),(b,a,b,a))\in X_1$, we have
$y=((a,b-1,a,b),(b,a,b-1,a))$, 
$y'=((a,b-1,a-1,b),(b-1,a,b-1,a-1))$,
$\KF{1}y=((a+1,b-1,a,b),(b-1,a+1,b-1,a))$, 
$\KF{1}y'=((a+1,b-1,a-1,b),(b-1,a-1,b,a-1))$.
Thus, we get $\Delta''=(0,1)$. Finally, we can check
$z=\KE{2}\KE{1}^2\KE{2}\KE{1}\KE{2}\KE{1}x=\KE{2}\KE{1}^3\KE{2}^2\KE{1}x=\KE{1}\KE{2}^2\KE{1}^3\KE{2}x=\KE{1}\KE{2}\KE{1}\KE{2}\KE{1}^2\KE{2}x=((a-1,b-1,a-1,b-1),(b-1,a-1,b-1,a-1))$ and we have
$\KF{1}z=((a,b-1,a-1,b-1),(b-1,a-1,b-1,a))$, 
$\KF{2}z=((a-1,b-1,a-1,b),(b,a-1,b-1,a-1))$.
Thus, we get $\Delta'=(1,2)$ as desired.

\subsubsection{}
For $x=((a,b,a,c),(b,a,c,a+2b-2c))\in X_2$, we have
$y=((a,b-1,a,c),(b-1,a,c,a+2b-2c-2))$, 
$y'=((a,b-1,a-1,c),(b-1,a-1,c,a+2b-2c-2))$, 
$\KF{1}y=((a+1,b-1,a,c),(b-1,a,c,a+2b-2c-1))$, 
$\KF{1}y'=((a+1,b-1,a-1,c),(b-1,a-1,c,a+2b-2c-1))$.
Thus, we get $\Delta''=(1,1)$. 
Finally, we can check $\KE{1}\KE{2}\KE{1}\KE{2}\KE{1}x=\KE{2}\KE{1}^3\KE{2}x=y'$ and 
we have $\KF{2}y'=((a-1,b-1,a,c),(b,a-1,c,a+2b-2c-2))$. 
Thus, we get $\Delta^f_{\varphi}(2,1,y')=1$ as desired.

\subsubsection{}
For $x=((a,b,c,a+b-c),(b,a,b,c))\in X_3$, we have
$y=((a,b-1,c,a+b-c),(b,a,b-1,c))$, 
$y'=((a-1,b-1,c,a+b-c-1),(b,a-1,b-1,c))$, 
$\KF{1}y=((a+1,b-1,c,a+b-c),(b-1,a+1,b-1,c))$, 
$\KF{1}y'=((a,b-1,c,a+b-c-1),(b-1,a,b-1,c))$.
Thus, we get $\Delta''=(0,0)$.
Finally, we can check $\KF{2}y'=\KE{1}y=((a-1,b-1,c,a+b-c),(b+1,a-1,b-1,c))$ and
we have $\KF{1}^2y'=((a+1,b-1,c,a+b-c-1),(b-1,a-1,b,c))$,
$\KF{2}\KF{1}^2y'=((a-1,b,c,a+b-c-1),(b,a-1,b,c))$.
Thus, we get $\Delta^f_{\varphi}(2,1,y')=2$, $\Delta^f_{\varphi}(2,1,\KF{1}^2y')=0$ as desired.

\subsection{Proof of Proposition \ref{kakunin2}}\label{pkakunin2}
It is enough to show
\begin{align*}
{} &\{m\in B(\lambda)\mid\varepsilon_1(m)\geq 2,\varepsilon_2(m)>0,
(\Delta^e_{\varepsilon}(1,2,m),\Delta^e_{\varepsilon}(2,1,m))=(1,1)\}\\
&=
\{((a,b,a+1,c),(b+1,a,c,a+2b-2c+1))\mid a\geq 2,0\leq c\leq b\}\cap B(\lambda)
\end{align*}
since we have $z=\KE{1}\KE{2}^2\KE{1}^2m=\KE{1}\KE{2}\KE{1}\KE{2}\KE{1}m=\KE{2}\KE{1}^3\KE{2}m=((a-2,b+1,a-2,c),(b+1,a-2,c,a+2b-2c))$
for $m=((a,b,a+1,c),(b+1,a,c,a+2b-2c+1))$ in the right hand side.

The inclusion $\supseteq$ is verified by direct calculation. To prove the reverse inclusion $\subseteq$, it is enough to show
$a_3\geq a_1$ for any $m=((a_1,a_2,a_3,a_4),(x_1,x_2,x_3,x_4))\in B(\lambda)$ in the left right hand 
because Corollary \ref{greaterthanorequaltosub1} implies $a_3=a_1+1,a_2\geq a_4$.
Assume $a_1>a_3$. 
By Corollary \ref{greater2} we have $x_1\leq x_3$.
Then, Corollary \ref{greaterthanorequaltosub2} implies $x_1=x_3$ and $x_2\geq x_4$ that means $m\in X_1\sqcup X_3$ (see \S\ref{pkakunin1}).
This contradicts $\Delta^e_{\varepsilon}(2,1,m)=1$.

\subsection{Proof of Proposition \ref{kakunin3}}\label{pkakunin3}
It is enough to show
\begin{align*}
{} &\{m\in B(\lambda)\mid\varepsilon_1(m)\geq 2,\varepsilon_2(m)>0,\varepsilon_2(\KE{1}^2m)>0,\Delta^e_{\varepsilon}(2,1,\KE{1}^2m)=0,
(\Delta^e_{\varepsilon}(1,2,m),\Delta^e_{\varepsilon}(2,1,m))=(0,2)\}\\
&=
\{((a,b,c,a+b-c-1),(b,a-2,b+1,c))\mid a\geq 2,b\geq1,0\leq c\leq a-2\}\cap B(\lambda)
\end{align*}
since we have $z=\KE{1}\KE{2}^2\KE{1}^2m=\KE{2}\KE{1}^2\KE{2}\KE{1}m=\KE{2}\KE{1}^3\KE{2}m=((a-1,b-1,c,a+b-c-2),(b-1,a-1,b-1,c))$
for $m=((a,b,c,a+b-c-1),(b,a-2,b+1,c))$ in the right hand side. 

The inclusion $\supseteq$ is verified by direct calculation. To prove the reverse inclusion $\subseteq$, it is enough to show
$x_3\geq x_1,x_2\geq x_4$ for any $m=((a_1,a_2,a_3,a_4),(x_1,x_2,x_3,x_4))\in B(\lambda)$ in the left hand side
because the following deduces $x_3=x_1+1$.
\begin{enumerate}
\item $x_3=x_1$, $x_2\geq x_4$ implies $m\in X_1\sqcup X_3$ (see \S\ref{pkakunin1}) and contradicts \mbox{$\Delta^e_{\varepsilon}(1,2,m)=0$.}
\item Let $x_3=x_1+n$ and assume $n\geq 2$ (then, we get a contradiction as (3)--(5)).
\item By Lemma \ref{lessthanorequalto4}, $(a_1,a_2,a_3,a_4)=(x_2+2n,x_1,x_4,x_1+n+x_2-x_4)$.
\item Because $a_2-(a_4+a_3-(a_1-2))=n-2\geq 0$ and $a_4+(a_3-(a_1-2))/2-a_2=1+(x_2-x_4)/2\geq 0$, we have
$\KE{1}^2m=((a_1-2,a_2,a_3,a_4),(x_1,x_2+2,x_1+n-2,x_4))$ by Lemma \ref{lessthanorequalto1}.
\item Because $x_1-1,x_1\leq x_1+n-2$ we see $\Delta^e_{\varepsilon}(2,1,\KE{1}^2m)=2$ by Lemma \ref{lessthanorequalto4}.
\end{enumerate}

In the rest, we show $x_3\geq x_1,x_2\geq x_4$.

First, we show $a_1>a_3$ as follows.
Corollary \ref{greaterthanorequaltosub1} and $\Delta^e_{\varepsilon}(2,1,m)=2$ imply $a_3\leq a_1$.
If $a_1=a_3$, then $a_2\geq a_4$ again by  Corollary \ref{greaterthanorequaltosub1} and $\Delta^e_{\varepsilon}(2,1,m)=2$.
It means $m\in X_1\sqcup X_2$ (see \S\ref{pkakunin1}) and contradicts $\Delta^e_{\varepsilon}(1,2,m)=0$.

Next, we show $x_3\geq x_1$. For this purpose, we assume $x_3<x_1$ (and $a_1>a_3$) to draw contradictions.
By Lemma \ref{lessthanorequalto2}, 
$a_1>a_3$ only happens when $x_2\geq x_4+x_3-x_1$ or $x_2\leq x_4+2(x_3-x_1)$.
In the former case, $x_2\geq x_4+x_3-(x_1-1)$ also holds because $x_2=a_1>a_3=x_4$ (and $x_1>x_3$).
Again, Lemma \ref{lessthanorequalto2} implies $\Delta^e_{\varepsilon}(2,1,m)=x_2-x_2=0$.
In the latter case, 
we may assume $a_1-2>a_3$ because otherwise 
\begin{align*}
\Delta^e_{\varepsilon}(2,1,\KE{1}^2m) 
=
\max(0,2+(a_1-2)-a_3+2a_2-2\max(a_2,a_4))=a_1-a_3>0
\end{align*}
follows from Corollary \ref{greaterthanorequaltosub1} and $a_4=x_3<x_1=a_2$.
Thus, we know $\KE{1}^2m=((a_1-2,a_2,a_3,a_4),(x_1,x_2,x_3,x_4-2))$
by Lemma \ref{lessthanorequalto1} and $a_4+(a_3-(a_1-2))/2-a_2=(x_2-x_4+2)/2\leq 0$.
This implies $\Delta^e_{\varepsilon}(2,1,\KE{1}^2m)=2$ since $x_2\leq (x_4-2)+2(x_3-(x_1-1))$ and Lemma \ref{lessthanorequalto2}.
In both cases, we arrived at contradictions.

Finally, we show $x_2\geq x_4$. For this purpose, we assume $x_2<x_4$ (and $x_3\geq x_1,a_1>a_3$) to draw contradictions. 
Note that in Lemma \ref{lessthanorequalto1} $x_2<x_4$ only occurs when $a_2>a_4+(a_3-a_1)/2$.
In each of the following, we arrived at a contradiction.
\begin{itemize}
\item Assume $a_1-2\geq a_3$. Because $a_2\geq a_4+(a_3-(a_1-2))/2$, again by Lemma \ref{lessthanorequalto1},
we have $\KE{1}^2m=((a_1-2,a_2,a_3,a_4),(x_1,x_2,x_3,x_4-2))$.
Lemma \ref{lessthanorequalto4} and $x_1-1,x_1\leq x_3$ imply $\Delta^e_{\varepsilon}(2,1,\KE{1}^2m)=2$.
\item Assume $a_1-2<a_3$. This only happens when $a_1=a_3+1$. Thanks to Lemma \ref{lessthanorequalto4}, $m$ is of the form
$m=((x_2+1,x_1,x_2,x_1),(x_1,x_2,x_1,x_2+1))$.
By Lemma \ref{lessthanorequalto3}, $\KE{1}^2m=((x_2-1,x_1,x_2,x_1),(x_1+1,x_2-1,x_1,x_2))$.
Thus, we have $\KE{2}\KE{1}^2m=((x_2,x_1,x_2-1,x_1),(x_1,x_2-1,x_1,x_2))$ and $\Delta^e_{\varepsilon}(2,1,\KE{1}^2m)=1$.
\end{itemize}

\section{Proof of Theorem \ref{maintheorem} : Uniqueness}\label{uniqconf}
The following is a version of ~\cite[Proposition 1.2,Remark 1.3.(a)]{Ste}.

\begin{Prop}\label{confprop}
Let $X$ be a good $I$-colored directed graph with a maximum element $x_0\in X$ 
and with homogeneous local confluence property (see Definition \ref{goodgraph},\ref{maxelem},\ref{localhomogconf}).
Then, for $x=\KF{i_1}\cdots\KF{i_s}x_0$, $\WT(x)=\sum_{k=1}^{s}i_k\in\mathbb{N}[I]$ is well-defined.
\end{Prop}

Here $\mathbb{N}[I]$ is the free commutative monoid generated by $I$.

\begin{proof}
We prove by induction on $d=\DIST(x):=\min\{s\geq 0\mid \exists(i_1,\cdots,i_s)\in I^s,x=\KF{i_1}\cdots\KF{i_s}x_0\}$.
$d=0$ is equivalent to $x=x_0$. Note $\WT(x_0)=0$ since $x_0$ is a maximum element. 

Let $d\geq 1$ and fix $(i_1,\cdots,i_d)\in I^d$ with $x=\KF{i_1}\cdots\KF{i_d}x_0$.
Take arbitrary ($e\geq d$ and) $(j_1,\cdots,j_e)\in I^e$ such that $x=\KF{j_1}\cdots\KF{j_e}x_0$.
If $i_1=j_1$, then we have $d=e$ and $\sum_{k=2}^{d}i_k=\sum_{k=2}^{e}j_k$ because $\DIST(x':=\KF{i_2}\cdots\KF{i_d}x_0)<d$. 
Otherwise, thanks to the homogeneous local confluence, there exist $t\geq 2$ and $(p_1,\cdots,p_t),(q_1,\cdots,q_t)\in I^t$ such that
$p_1=i_1,q_1=j_1,\sum_{k=1}^{t}p_k=\sum_{k=1}^{t}q_k,\KE{p_t}\cdots\KE{p_1}x=\KE{q_t}\cdots\KE{q_1}x=:z\in X$.
By induction hypothesis, we have $\WT(z)+\sum_{k=2}^{t}p_k=\WT(\KE{p_1}x)$ and $\DIST(z)=\DIST(\KE{p_1}x)-(t-1)$.
Thus we can apply induction hypothesis to $\KE{q_1}x$ and we have $\WT(\KE{q_1}x)=\sum_{k=2}^{e}j_k=\WT(z)+\sum_{k=2}^{t}q_k$
that implies $\sum_{k=1}^{d}i_k=\sum_{k=1}^{e}j_k$.
\end{proof}

\begin{Rem}\label{wkc}
Continue on from Proposition \ref{confprop} and assume that $X$ satisfies (S2) further.
Fix $\lambda\in P^+$ such that $\forall i\in I,\langle h_i,\lambda\rangle=\varphi_i(x_0)$.
By induction on $\DIST(x)$,  Proposition \ref{confprop} implies that
the graph satisfies (K1) (see \S\ref{defkc}) by defining $\WTT(x)=\lambda-U(\WT(x))$ for $x\in X$, where $U:\mathbb{N}[I]\to P,\sum_{k}i_k\mapsto \sum_{k}\alpha_{i_k}$.
\end{Rem}

The following uniqueness result is similar to ~\cite[Proposition 1.4]{Ste}.

\begin{Prop}\label{uniqprop}
For a symmetrizable GCM $A=(a_{ij})_{i,j\in I}$
with $\forall i\ne\forall j\in I, A|_{i,j}=A_1\oplus A_1,A_2,B_2,\TRANS{B_2}$,
let $X$, $X'$ be $A$-regular graphs satisfying 
\begin{align*}
\forall i\ne\forall j\in I,A|_{i,j}=B_2\Rightarrow \textup{(S6),(S7),(S8),(S9)}
\end{align*}
with maximal elements $x_0\in X$, $x'_0\in X'$ respectively. 
If $\varphi_i(x)=\varphi_i(x')$ for all $i\in I$, there exists a unique $I$-colored directed graph isomorphism $X\ISOM X'$.
\end{Prop}

\begin{proof}
Uniqueness is obvious because $\forall x\in X,\exists (i_1,\cdots,i_s)\in I^s,\KF{i_1}\cdots\KF{i_s}x_0=x$.
To prove existence, we will construct a bijection $\EFU_k:X_k\ISOM X'_k$ such that
\begin{enumerate}
\item[(1$_k$)] $\bigsqcup_{\ell=0}^{k}\EFU_{\ell}:\bigsqcup_{\ell=0}^{k}X_{\ell}\ISOM\bigsqcup_{\ell=0}^{k}X'_{\ell}$ 
is a $I$-colored 
directed graph isomorphism,
\item[(2$_k$)] $\forall x\in X_k,\forall i\in I,\varphi_i(x)=\varphi_i(\EFU_{k}(x)),\varepsilon_i(x)=\varepsilon_i(\EFU_{k}(x))$.
\end{enumerate}
by induction on $k$, where \mbox{$X_k=\{x\in X\mid \DIST(x)=k\},X'_k=\{x\in X'\mid \DIST(x)=k\}$.}

For $k=0$, the only choice is $\EFU_0(x_0)=x'_0$ and (2$_0$) is trivial.
For $k\geq1$, we define $\EFU_k(x)=\KF{i}\EFU_{k-1}(\KE{i}x)$ if $\KE{i}x\ne\ZERO$.
$\EFU_k(x)$ is well-defined because of (X),(Y),(Z).
\begin{itemize}
\item[(X)] $\forall x\in X_{k},\exists i\in I,\KE{i}x\in X_{k-1}$ by Proposition \ref{confprop}.
\item[(Y)] $\KF{i}\EFU_{k-1}(\KE{i}x)\ne\ZERO$ because $\varphi_i(\EFU_{k-1}(\KE{i}x))=\varphi_i(\KE{i}x)>0$ by (2$_{k-1}$).
\item[(Z)] For $i\ne j\in I$ with $\KE{i}x\ne\ZERO\ne\KE{j}x$, we will show $\KF{i}\EFU_{k-1}(\KE{i}x)=\KF{j}\EFU_{k-1}(\KE{j}x)$ case by case
as follows.
\end{itemize}

When $A|_{i,j}=A_1\oplus A_1,A_2$, (Z) is in the proof of ~\cite[Proposition 1.4]{Ste} (or similar to the arguments below).
So let us $A|_{i,j}=B_2$. By (S2),(S3), possibilities of $\Delta(x)=(\Delta^e_{\varepsilon}(i,j,x),\Delta^e_{\varepsilon}(j,i,x))$ are $\Delta(x)=(0,0),(1,0),(0,1),(1,1),(0,2),(1,2)$\footnote{As ~\cite[p.4821]{Ste}, cases $(0,0),(0,1)$ never occur by ~\cite[Lemma 2.5]{Ste}. Yet we do not need it.}.
Among them, cases $\Delta(x)=(0,0),(1,0),(0,1),(1,1),(0,2)$, (Z) is again the same as in the proof of ~\cite[Proposition 1.4]{Ste} (or similar to the arguments below).
Thus, we assume $\Delta(x)=(1,2)$. 
By (D$^-$) in (S6), $\exists y=\KE{i}^2\KE{j}x\in X_{k-3},\exists y'=\KE{i}^2\KE{j}^2\KE{i}x\in X_{k-5}$.
Again (S2),(S3) imply $\Delta''=(\Delta^f_{\varphi}(i,j,y),\Delta^f_{\varphi}(i,j,y'))=(0,0),(1,0),(0,1),(1,1)$.

Assume $\Delta''=(0,1)$. By (Q$^-_1$) in (D$^-$) in (S6), 
we have $\exists z=\KE{j}\KE{i}^3\KE{j}^2\KE{i}x=\KE{i}\KE{j}^2\KE{i}^3\KE{j}x\in X_{k-7},\Delta'(z)=(1,2)$
and as in Remark \ref{luck} $(\Delta^e_{\varepsilon}(i,j,\KF{i}^2\KF{j}z),\Delta^e_{\varepsilon}(i,j,\KF{i}^2\KF{j}^2\KF{i}z))=(0,1)$.
Then, by induction hypothesis and (S7), we have $\KF{j}\KF{i}^3\KF{j}^2\KF{i}\EFU_{k-7}(z)=\KF{i}\KF{j}^2\KF{i}^3\KF{j}\EFU_{k-7}(z)$.
Since $\EFU_{k-1}(\KE{i}x)=\KF{j}^2\KF{i}^3\KF{j}\EFU_{k-7}(z)$ and $\EFU_{k-1}(\KE{j}x)=\KF{i}^3\KF{j}^2\KF{i}\EFU_{k-7}(z)$, we are done.

Assume $\Delta''=(0,0)$. By (R$^-$) in (D$^-$) in (S6), 
we have $\KF{j}\KF{i}^2y'\ne\ZERO,\KF{j}y'=\KE{i}y$ and $\Delta^f_{\varphi}(j,i,\KF{i}^2y')=0,\Delta'=(0,2)$, where $\Delta'=(\Delta^f_{\varphi}(i,j,y'),\Delta^f_{\varphi}(j,i,y'))$. 
By induction hypothesis and (S9), we have $\KF{i}\KF{j}^2\KF{i}^2\EFU_{k-5}(y')=\KF{j}\KF{i}^3\KF{j}\EFU_{k-5}(y')=\KF{j}\KF{i}^2\EFU_{k-3}(y)$.
Since $\EFU_{k-1}(\KE{i}x)=\KF{j}^2\KF{i}^2\EFU_{k-5}(y')$ and $\EFU_{k-1}(\KE{j}x)=\KF{i}^2\EFU_{k-3}(y)$, we are done.

Assume $\Delta''=(1,1)$. By (P$^-_1$) in (D$^-$) in (S6), 
we have $\KF{i}^2y'\ne\ZERO$, $\KF{j}y'=\KE{i}y$, $\Delta'=(1,1)$.
Apply induction hypothesis and (S8), the rest is the same. 


Because $\Delta''\ne(1,0)$ by (D$^-$) in (S6), (Z) is proved and thus $\EFU_k$ is well-defined.

Finally, we show (1$_k$) and (2$_k$).
$\EFU_k$ is surjective because for any $x'\in X'_{k}$ with $\KE{i}x'=y'\in X'_{k-1}$ we have $\EFU_{k}(\KF{i}\EFU_{k-1}^{-1}(y'))=x'$.
By symmetry we now know that $|X_k|= |X'_{k}|$ and $\EFU_k$ is a bijection. 
For (2$_k$), by (1$_k$) we have $\forall x\in X_k,\forall i\in I,\varepsilon_i(x)=\varepsilon_i(\EFU_{k}(x))$.
Then, $\forall x\in X_k,\forall i\in I,\varphi_i(x)=\varphi_i(\EFU_{k}(x))$ follows from Remark \ref{wkc} because
by construction $\EFU_{k}$ is $\WT$-preserving.
\end{proof}



\end{document}